\documentclass[11pt]{article}
\usepackage{amsmath,amsthm,amssymb}

\def\titlerunning#1{\gdef\titrun{#1}}
\makeatletter
\def\author#1{\gdef\autrun{\def\and{\unskip, }#1}\gdef\@author{#1}}
\def\address#1{{\def\and{\\\hspace*{18pt}}\renewcommand{\thefootnote}{}%
\footnote {#1}}%
\markboth{\autrun}{\titrun}}
\makeatother
\def\email#1{\hspace*{4pt}{\em e-mail}: #1}
\def\MSC#1{{\renewcommand{\thefootnote}{}%
\footnote{\emph{Mathematics Subject Classification (2010):} #1}}}
\def\keywords#1{\par\medskip
\noindent\textbf{Keywords:} #1}

\newtheorem{theorem}{Theorem}[section]
\newtheorem{prop}[theorem]{Proposition}
\newtheorem{cor}[theorem]{Corollary}
\newtheorem{lemma}[theorem]{Lemma}

\newtheorem{defin}[theorem]{Definition}

\theoremstyle{definition}

\newtheorem{remark}[theorem]{Remark}

\numberwithin{equation}{section}

\def\cL{\mathcal L}

\def\cA{\mathcal A}
\def\cC{\mathcal C}
\def\cB{\mathcal B}
\def\cD{\mathcal D}
\def\cE{\mathcal E}

\def\cI{\mathcal I}

\def\cO{\mathcal O}
\def\cP{\mathcal P}
\def\cX{\mathcal X}
\def\cY{\mathcal Y}

\def\cT{\mathcal T}

\def\cS{\mathcal S}

\def\cQ{\mathcal Q}

\def\PG{{\rm PG}}

\def\GF{{\rm GF}}

\def\PGL{{\rm PGL}}
\def\PSL{{\rm PSL}}

\def\GL{{\rm GL}}

\begin{document}


\baselineskip=16pt

\titlerunning{}

\title{Cameron--Liebler line classes of $\PG(3,q)$ admitting $\PGL(2,q)$}

\author{Antonio Cossidente
\and
Francesco Pavese}

\date{}

\maketitle

\address{A. Cossidente: Dipartimento di Matematica, Informatica ed Economia, Universit{\`a} degli Studi della Basilicata, Contrada Macchia Romana, 85100, Potenza, Italy;
\email{antonio.cossidente@unibas.it}
\and
F. Pavese: Dipartimento di Meccanica, Matematica e Management, Politecnico di Bari, Via Orabona 4, 70125 Bari, Italy; 
\email{francesco.pavese@poliba.it}}

\bigskip

\MSC{Primary  51E20; Secondary  05B25 05E30}


\begin{abstract}
In this paper we describe an infinite family of Cameron--Liebler line classes of $\PG(3,q)$ with parameter $(q^2+1)/2$, $q \equiv 1 \pmod{4}$. The example obtained admits $\PGL(2,q)$ as an automorphism group and it is shown to be isomorphic to none of the infinite families known so far whenever $q \ge 9$. 

\end{abstract}

\section{Introduction}\label{sect:intro}

Let $\PG(3,q)$ be the three dimensional projective space over $\GF(q)$. A {\em spread} of $\PG(3,q)$ is a set of $q^2+1$ pairwise disjoint lines. The subject of this paper is the study of certain Cameron--Liebler line classes of $\PG(3,q)$. 
	\begin{defin}[\cite{CL}, \cite{P}]	
In $\PG(3,q)$, a {\em Cameron--Liebler line class} $\cal L$  with parameter $x$ is a set of lines such that every spread of $\PG(3,q)$ contains exactly $x$ lines of $\cal L$.
	\end{defin}
The notion of Cameron--Liebler line class was introduced in the seminal paper \cite{CL} in order to classify those collineation groups of $\PG(3,q)$ having the same number of orbits on points and lines. On the other hand, a classification of Cameron--Liebler line classes would yield a classification of symmetric tactical decompositions of points and lines of $\PG(n,q)$ and that of certain families of weighted point sets of $\PG(3,q)$ \cite{CL},\cite{PJC}, \cite{P}. Cameron--Liebler line classes are also related to other combinatorial structures such as two--intersection sets, strongly regular graphs and projective two--weight codes \cite{BKLP}, \cite{MR}. Recently in \cite{FI}, it has been pointed out that Cameron--Liebler line classes of $\PG(3,q)$ are equivalent to other well studied objects, namely the Boolean degree $1$ functions on the Grassmann graph $J_q(4,2)$. 

The size of a  Cameron--Liebler line class $\cal L$ equals $x(q^2+q+1)$ and the number of lines of $\cL$ intersecting a given line $\ell$ only depends on whether $\ell$ belongs to $\cL$ or does not. Note that, under the Klein correspondence between the lines of $\PG(3,q)$  and points of a Klein quadric ${\cal Q}^+(5,q)$,  a Cameron--Liebler line class with parameter $i$ is equivalent to a so--called {\em $i$--tight set} of ${\cal Q}^+(5,q)$.
	\begin{defin}
A set $\cT$ of points of ${\cal Q}^+(5,q)$ is said to be {\em $i$--tight} if
$$
\vert P^\perp\cap{\cal T}\vert=
\begin{cases} i(q+1)+q^2 & if \quad P\in {\cal T}\\
i(q + 1) & if \quad P\not\in {\cal T}\end{cases},
$$
where $\perp$ denotes the polarity of $\PG(5,q)$ associated with ${\cal Q}^+(5,q)$. 
	\end{defin}	
For details and results on tight sets of polar spaces, see \cite{BKLP}. There exist trivial examples of Cameron--Liebler line classes $\cal L$ with parameters $x=1,2$ and $x=q^2,q^2-1$. A Cameron--Liebler line class with parameter $x=1$ is either the set of lines through a point or the set of lines in a plane. A Cameron--Liebler line class with parameter $x=2$ is the union of the two previous examples, if the point is not in the plane \cite{CL}, \cite{P}. In general, the complement of a Cameron--Liebler line class with parameter $x$ is a Cameron--Liebler line class with parameter $q^2+1-x$ and the union of two disjoint Cameron--Liebler line classes with parameters $x$ and $y$, respectively, is a Cameron--Liebler line class with parameter $x+y$. 

It was conjectured that no other examples of Cameron--Liebler line classes exist \cite{CL}. Bruen and Drudge \cite{BD} were the first to find a counterexample, namely an infinite family of Cameron--Liebler line classes with parameter $x=(q^2+1)/2$, $q$ odd. See Section \ref{known} for a description of the other known infinite families of Cameron--Liebler line classes with parameter $(q^2+1)/2$. For non--existence and classification results of Cameron--Liebler line classes we refer the reader to \cite{GM}, \cite{M}, \cite{GM1}. 

In this paper we will introduce a new derivation technique for Cameron--Liebler line classes with parameter $(q^2+1)/2$, which generalizes \cite[Theorem 3.9]{CP1}, see Theorem \ref{cl}. Applying such a derivation to the Bruen--Drudge's example, we construct a new family of Cameron--Liebler line classes with parameter $(q^2+1)/2$, $q \equiv 1 \pmod{4}$, $q\ge 9$ odd, not equivalent to the examples known so far and admitting an automorphism group isomorphic to $\PGL(2,q)$.

Throughout the paper $q$ is a power of an odd prime.

\section{The known examples of Cameron--Liebler line classes with parameter $\frac{q^2+1}{2}$}\label{known}

Up to date, the following infinite families of Cameron--Liebler line classes with parameter $(q^2+1)/2$ are known: 
\begin{enumerate}
\item[1)] the Bruen--Drudge's family \cite{BD}, admitting the group $K \simeq \PSL(2,q^2)$ stabilizing an elliptic quadric $\cQ^-(3,q)$ of $\PG(3,q)$, $q$ odd.

Let $\cQ^-(3,q)$ be an elliptic quadric of $\PG(3,q)$ with quadratic form $Q$. The points of $\PG(3,q) \setminus \cQ$ can be partitioned in two sets of equal size, say $O_s$, $O_n$, corresponding to points of $\PG(3,q)$ such that the evaluation of the quadratic form $Q$ is a non--zero square or a non--square in $\GF(q)$, respectively. A tangent line to $\cQ^-(3,q)$ contains either $q$ points of $O_s$ or $q$ points of $O_n$. In particular, the set of lines that are tangent to $\cQ^-(3,q)$ can be partitioned in two sets of equal size, say $\cT_s$, $\cT_n$, where a tangent $t$ belongs to $\cT_s$ or $\cT_n$, depending on $t$ contains $q$ points of $O_s$ or $q$ points of $O_n$. Let $\cS_1$, $\cS_2$ be the set of lines that are external or secant to $\cQ^-(3,q)$, respectively. Then gluing together one set among $\cT_s$, $\cT_n$ and one set among $\cS_1$, $\cS_2$, an example of Cameron--Liebler line class constructed by Bruen and Drudge in \cite{BD} is obtained. In particular, if ${\cal L'} = \cT_s \cup \cS_1$, then $\cL'$ has the following three characters with respect to line--sets in planes of $\PG(3,q)$: 
$$
\frac{q^2+q}{2}-q, \frac{q^2+q}{2}+1, q^2+\frac{q+1}{2}, 
$$
and 
$$
\frac{q+1}{2}, \frac{q^2+q}{2}, \frac{q^2+q}{2}+q+1,
$$
with respect to line--stars of $\PG(3,q)$.

\item[2)] the first family derived from Bruen--Drudge \cite{CP}, \cite{GMP}, admitting the stabilizer $K'$ of a point of $\cQ^-(3,q)$ in $K$, $q \ge 5$ odd.

Consider a point $R$ of $\cQ^-(3,q)$ and let $\rho$ be the tangent plane to $\cQ^-(3,q)$ at the point $R$. Let $\cL''$ be the line--set of $\PG(3,q)$ obtained from $\cL'$, by replacing the $q^2$ lines of $\cS_1$ contained in $\rho$ with the $q^2$ lines of $\cS_2$ passing through $R$. Then ${\cal L}''$ is again a Cameron--Liebler line class with parameter $(q^2+1)/2$. In particular, ${\cal L}''$ has the following five characters with respect to line--sets in planes of $\PG(3,q)$:
$$
\frac{q+1}{2}, \frac{q^2+q}{2}-(q+1), \frac{q^2+q}{2}, \frac{q^2+q}{2}+q+1, q^2+\frac{q-1}{2},
$$
and 
$$
\frac{q+3}{2}, \frac{q^2+q}{2}-q, \frac{q^2+q}{2}+1, \frac{q^2+q}{2}+q+2, q^2+\frac{q+1}{2},
$$ 
with respect to line--stars of $\PG(3,q)$. It turns out that, if $q >3$, these characters are distinct from those of a Bruen--Drudge Cameron--Liebler line class.

\item[3)] the second family derived from Bruen--Drudge \cite{CP1}, say $\cL'''$, admitting a subgroup of $K'$ of order $q^2(q+1)$, $q \ge 7$ odd.

Here the existence of a pencil of elliptic quadrics fixed by a subgroup of $K'$ of order $q^2(q+1)$ plays a crucial role and the derivation is similar to the previous example with a more restrictive selection of tangent lines to the elliptic quadrics of the pencil. The characters of the Cameron--Liebler line class $\cL'''$ with respect to line--sets in planes of $\PG(3,q)$ form a subset of:
$$
\left\{ q^2 + \frac{q+1}{2}, q^2 - \frac{3(q+1)}{2}, \frac{q^2+q}{2}+2q+3, \frac{q^2+q}{2}+q+2, \right.
$$
$$
\left. \frac{q^2+q}{2}+1, \frac{q^2+q}{2}-q, \frac{q^2+q}{2}-2q-1, \frac{q^2+q}{2}-2(q+1) \right\} .
$$ 
and with respect to line--stars of $\PG(3,q)$ form a subset of:
$$
\left\{ \frac{q+1}{2}, \frac{5(q+1)}{2}, \frac{q^2+q}{2}-2(q+1), \frac{q^2+q}{2}-(q+1), \right.
$$
$$
\left. \frac{q^2+q}{2}, \frac{q^2+q}{2}+q+1, \frac{q^2+q}{2}+2(q+1), \frac{q^2+q}{2}+3(q+1) \right\},
$$ 
\item[4)] the ``cyclic'' family \cite{DDMR}, \cite{FMX}, admitting a group of order $3(q-1)(q^2+q+1)/2$, $q \equiv 5$ or $9 \pmod{12}$. 

Infinite families of Cameron--Liebler line classes with parameter $(q^2-1)/2$ were found for $q \equiv 5$ or $9\pmod{12}$ in \cite{DDMR}, \cite{FMX}. By construction, for a line class $\cal X$ of such a family there is a fixed plane $\Pi$ and a fixed point $z\not\in\Pi$ such that $\cX$ never contains the lines $\cal Y$ of the plane $\Pi$ and the lines $\cal Z$ through the point $z$. Therefore, ${\cal X}\cup{\cal Y}$ and ${\cal X}\cup{\cal Z}$  are both examples of  Cameron--Liebler line classes with parameter $(q^2+1)/2$. In particular both examples $\cX \cup \cY$ and $\cX \cup {\cal Z}$ admit $q^2+q+1$ as a character.
\end{enumerate}

\section{A new family}

\subsection{The geometric setting} \label{setting}

Let $X_1,\dots, X_4$ be homogeneous projective coordinates in $\PG(3,q)$. Let $\cQ_{\lambda}$ be the quadric with equation $X_1X_3 - X_2^2 + \lambda X_4^2 = 0$, $\lambda\in\GF(q)$. Then the non--degenerate quadrics $\cQ_{\lambda}$, $\lambda\in\GF(q)$, together with the plane $\pi$ satisfying the equation $X_4=0$, form a pencil $\cal P$ of $\PG(3,q)$. The base locus of $\cal P$ is the conic $\cC = \{(1,t,t^2,0) \;\; | \;\; t \in \GF(q)\} \cup \{U_3\}$, where $U_i$ denotes the points having $1$ in the $i$-th position and $0$ elsewhere. The quadric $\cQ_\lambda$ is elliptic or hyperbolic according as $\lambda$ is a non--square or a non--zero square in $\GF(q)$, while $\cQ_0$ is a quadratic cone having as vertex the point $U_4$ and as base the conic $\cC$. If $\lambda \ne 0$, let $\perp_{\lambda}$ be the orthogonal polarity associated to $\cQ_\lambda$. Note that $U_4^{\perp_{\lambda}}=\pi$, $\forall \lambda \in \GF(q) \setminus \{0\}$. 

There exists a group $G$ of order $q^3-q$ that is isomorphic to $\PGL(2,q)$ which stabilizes each of the quadrics of the pencil $\cP$. Here and in the sequel we shall find it helpful to work with the elements of $\PGL(4, q)$ as matrices in $\GL(4, q)$. We shall consider the points of $\PG(3,q)$ as column vectors, with matrices acting on the left.

\begin{lemma}\label{pencil}
The group $G$ given by
$$
M = \left(
\begin{array}{cccc}
a^2 & 2ac & c^2 & 0 \\
ab & ad+bc & cd & 0 \\
b^2 & 2bd & d^2 & 0 \\
0 & 0 & 0 & ad-bc \\
\end{array}
\right),
$$
with $a, b, c, d \in \GF(q)$, $ad - bc \neq 0$, stabilizes $\cQ_\lambda$. 
\end{lemma}

\begin{remark}\label{tra}
The group $G$ acts faithfully on points of $\pi$. Hence $G$ is $3$--transitive on points of the conic $\cC$ and transitive on points of $\pi$ that are external or internal to $\cC$. Dually, $G$ is $3$--transitive on lines of $\pi$ that are tangent to $\cC$ and transitive on lines of $\pi$ that are secant or external to $\cC$.  
\end{remark}

\begin{lemma}\label{point}
The group $G$ has $q+4$ orbits on points of $\PG(3,q)$:
\begin{enumerate}
\item[1)] The point $U_4$,
\item[2)] the plane $\pi$ is partitioned into three $G$--orbits:
\begin{itemize}
\item the conic $\cC$,
\item an orbit $\cI$ of size $q(q-1)/2$ consisting of internal points,
\item an orbit $\cE$ of size $q(q+1)/2$ consisting of external points,
\end{itemize}
\item[3)] an orbit of size $q^2-1$ consisting of points of $\cQ_0 \setminus (\cC \cup \{U_4\})$,
\item[4)] $(q-1)/2$ orbits of size $q^2+q$ consisting of points of $\cQ_\lambda \setminus \cC$, $\lambda$ a non--zero square in $\GF(q)$,
\item[5)] $(q-1)/2$ orbits of size $q^2-q$ consisting of points of $\cQ_\lambda \setminus \cC$, $\lambda$ a non--square in $\GF(q)$.
\end{enumerate}
\end{lemma}
\begin{proof}
It is straightforward to prove {\em 1)} and {\em 2)}, see also Remark \eqref{tra}. Let $P = (1,0,0,\alpha) \in \cQ_0$, for some $\alpha \in \GF(q) \setminus \{0\}$. Then $P^M = P$ implies $b = 0$ and $a = d$. Hence $Stab_G(P) = \eta$, where
$$
\eta = \left\{ \left(\begin{array}{cccc}  1 & 2x & x^2 & 0 \\ 0 & 1 & x & 0 \\ 0 & 0 & 1 & 0 \\ 0 & 0 & 0 & 1 \end{array}\right) \;\; | \;\; x \in \GF(q) \right\} 
$$
and $|Stab_G(P)| = q$. It follows that $|P^G| = q^2-1 = |{\cQ_0} \setminus ({\cC} \cup \{U_4\})|$. Let $\lambda$ be a non--zero square in $\GF(q)$ and let $P = (0, \sqrt{\lambda},0,1) \in \cQ_\lambda$. Since the line joining $U_4$ and $P$ meets the plane $\pi$ in $U_2$, we have that $Stab_G(P) \le Stab_G(U_2)$. Note that $Stab_G(U_2)$ is a dihedral group of order $2(q-1)$ generated by 
$$
\phi = \left\{ \left(\begin{array}{cccc}  1 & 0 & 0 & 0 \\ 0 & d & 0 & 0 \\ 0 & 0 & d^2 & 0 \\ 0 & 0 & 0 & d \end{array}\right) \;\; | \;\; d \in \GF(q) \setminus \{0\} \right\}, \;\; \tau_1 =  \left(\begin{array}{cccc}  0 & 0 & 1 & 0 \\ 0 & 1 & 0 & 0 \\ 1 & 0 & 0 & 0 \\ 0 & 0 & 0 & -1 \end{array}\right) .
$$
On the other hand $\tau_1$ does not fix $P$ and hence $Stab_G(P) = \phi$. Therefore $|Stab_G(P)| = q-1$ and $|P^G| = q^2+q = |{\cQ_\lambda} \setminus {\cC}|$. Analogously, if $\lambda$ is a non--square in $\GF(q)$ and $P = (-s, 0, 1, \sqrt{s/\lambda}) \in \cQ_\lambda$, where $s$ is a non--square in $\GF(q)$, we have that the line joining $U_4$ and $P$ meets the plane $\pi$ in the point $R = (-s, 0,1,0)$ and again $Stab_G(P) \le Stab_G(R)$. Note that $Stab_G(R)$ is a dihedral group of order $2(q+1)$ generated by 
$$
\psi = \left\{ \left(\begin{array}{cccc}  a^2 & 2sab & s^2b^2 & 0 \\ ab & a^2+sb^2 & sab & 0 \\ b^2 & 2ab & a^2 & 0 \\ 0 & 0 & 0 & a^2-sb^2 \end{array}\right) \;\; | \;\; a,b \in \GF(q), a^2-sb^2 = \pm 1 \right\}, 
$$
$$
\tau_2 =  \left(\begin{array}{cccc}  1 & 0 & 0 & 0 \\ 0 & -1 & 0 & 0 \\ 0 & 0 & 1 & 0 \\ 0 & 0 & 0 & -1 \end{array}\right) , \;\; \tau_3 =  \left(\begin{array}{cccc}  0 & 0 & s^2 & 0 \\ 0 & s & 0 & 0 \\ 1 & 0 & 0 & 0 \\ 0 & 0 & 0 & -s \end{array}\right) .
$$
Indeed, there are $2(q+1)$ elements $(a,b) \in \GF(q) \times \GF(q)$ such that $a^2-sb^2 = \pm 1$ and there are either $4$ or $2$ couples inducing the identity collineation according as $-1$ is a square or not in $\GF(q)$. This means that $\psi$ is a group of order $(q+1)/2$ if $q \equiv 1 \pmod{4}$ or $q+1$ if $q \equiv -1 \pmod{4}$. In any case $|\langle \psi, \tau_3 \rangle| = q+1$. On the other hand, $\tau_2$ does not fix $P$ and hence $Stab_G(P) = \langle \psi, \tau_3 \rangle$. Therefore $|Stab_G(P)| = q+1$ and $|P^G| = q^2-q = |{\cQ_\lambda} \setminus {\cC}|$.
\end{proof}

From the proof of the previous lemma we have the following result.

\begin{cor}\label{cono}
A line $\ell$ through $U_4$ is secant to every elliptic quadric of $\cP$ or to every hyperbolic quadric of $\cP$, according as $\ell \cap \pi$ belongs to $\cI$ or to $\cE$. 
\end{cor}

	\begin{lemma}\label{tang}
Every line of $\PG(3,q)$ not contained in $\pi$ and not containing a point of $\cC$, is tangent to exactly one quadric $\cQ_{\lambda}$ of $\cal P$.
	\end{lemma}
	\begin{proof}
Let $\ell$ be a line of $\PG(3,q)$ not contained in $\pi$ and not containing a point of $\cC$. If $|\ell \cap \cQ_0| = 1$, then either $\ell$ contains the vertex $U_4$ and the result follows from Corollary \ref{cono}, or there exists a point $T \in \cC$ such that $\ell$ is contained in the plane $\sigma$ spanned by $U_4 T$ and the line of $\pi$ tangent to $\cC$ at $T$. If the latter case occurs, then either $\sigma \cap \cQ_\lambda = \{T\}$ and $\lambda$ is a non--square, or $\sigma \cap \cQ_\lambda$ consists of two lines through $T$ and $\lambda$ is a non--zero square. In any case, if $\lambda \ne 0$, then $|\ell \cap \cQ_\lambda| \in \{0,2\}$. 

Assume that $\ell \cap \cQ_{\bar \lambda} = \{P\}$, for some fixed non--zero element ${\bar \lambda} \in \GF(q)$. Then $\ell \subset P^{\perp_{\bar \lambda}}$. 

If ${\bar \lambda}$ is a non--square, then we may assume $P = (-s,0,1,\sqrt{s/{\bar \lambda}})$, where $s$ is a fixed non--square in $\GF(q)$. In this case the plane $P^{\perp_{\bar \lambda}}$ meets $\pi$ in a line $r$ external to $\cC$ and it meets $\cQ_{\lambda}$, $\lambda \ne {\bar \lambda}$, in a non--degenerate conic ${\cal C}_{\lambda}$, $\lambda \in \GF(q) \setminus \{{\bar \lambda}\}$. Then $P$, $r$, ${\cal C}_{\lambda}$, $\lambda \in \GF(q) \setminus \{{\bar \lambda}\}$, form a pencil of quadrics of $P^{\perp_{\bar \lambda}}$. From \cite[Table 7.7]{H1}, $r$ is the polar line of $P$ with respect to ${\cal C}_{\lambda}$. Hence, $P$ is an internal point with respect to ${\cal C}_{\lambda}$ and the result follows.

 If ${\bar \lambda}$ is a square, then we may assume $P = (0,\sqrt{{\bar \lambda}},0, 1)$. In this case the plane $P^{\perp_{\bar \lambda}}$ meets $\pi$ in a line $r'$ secant to $\cC$, with $\ell \cap \cC = \{R_1, R_2\}$, and it meets $\cQ_{\lambda}$, $\lambda \ne {\bar \lambda}$, in a non--degenerate conic ${\cal C}'_{\lambda}$, $\lambda \in \GF(q) \setminus \{{\bar \lambda}\}$. On the other hand, $\cQ_{\bar \lambda} \cap P^{\perp_{\bar \lambda}}$ is a degenerate quadric $\cD$ consisting of the two lines $P R_1$ and $P R_2$. Then $\cD$, $r'$, ${\cal C}'_{\lambda}$, $\lambda \in \GF(q) \setminus \{{\bar \lambda}\}$, form a pencil of quadrics of $P^{\perp_{\bar \lambda}}$. In particular, $r'$ is the polar line of $P$ with respect to ${\cal C}'_{\lambda}$ and the lines $P R_1$, $P R_2$ are tangent to $\cC'_{\lambda}$ at $R_1$ and $R_2$, respectively, for every $\lambda \in \GF(q) \setminus \{{\bar \lambda}\}$. Hence, every line of $P^{\perp_{\bar \lambda}}$ through $P$ and not containing $R_1$ and $R_2$ is either external or secant to $\cC'_{\lambda}$.
	\end{proof}

\begin{lemma}\label{lines}
The group $G$ has $3q+5$ orbits on lines of $\PG(3,q)$:
\begin{enumerate}
\item[1)] The lines of $\pi$ are partitioned into three $G$--orbits:
\begin{itemize}
\item $\cL_1$ consisting of the $q+1$ lines that are tangent to $\cC$,
\item $\cL_2$ consisting of the $q(q-1)/2$ lines that are external to $\cC$,
\item $\cL_3$ consisting of the $q(q+1)/2$ lines that are secant to $\cC$,
\end{itemize}
\item[2)] the lines through $U_4$ are partitioned into three $G$--orbits:
\begin{itemize}
\item $\cL_1'$ consisting of the $q+1$ lines of the cone $\cQ_0$,
\item $\cL_2'$ consisting of the $q(q-1)/2$ lines meeting $\pi$ in a point of $\cI$,
\item $\cL_3'$ consisting of the $q(q+1)/2$ lines meeting $\pi$ in a point of $\cE$,
\end{itemize}
\item[3)] $q-1$ orbits of size $q+1$, each of them is a regulus of $\cQ_\lambda$, $\lambda$ a non--zero square in $\GF(q)$,
\item[4)] $q-1$ orbits of size $(q^3-q)/2$ consisting of lines tangent to $\cQ_\lambda$, $\lambda \ne 0$, and meeting $\pi$ in exactly one point of $\cI$,
\item[5)] $q-1$ orbits of size $(q^3-q)/2$ consisting of lines tangent to $\cQ_\lambda$, $\lambda \ne 0$, and meeting $\pi$ in exactly one point of $\cE$,
\item[6)] $\cL_4$ consisting of the $q^3-q$ lines tangent to $\cQ_0$ not through $U_4$ and meeting $\pi$ in exactly one point of $\cE$,
\item[7)] $\cL_4'$ consisting of the $q^3-q$ lines that are secant to every quadric $\cQ_\lambda$ and meeting $\pi$ in exactly one point of $\cC$.
\end{enumerate}
\end{lemma}
\begin{proof}
It is straightforward to prove {\em 1)} and {\em 2)}, see also Remark \eqref{tra}. In order to prove {\em 3)}, let $\lambda$ be a non--zero square in $\GF(q)$ and let $\ell_1$ be the line joining $P=(0, \sqrt{\lambda}, 0, 1)$ and $U_1$. Then $\ell_1$ is a line of $\cQ_{\lambda}$. Note that $Stab_G(\ell_1) \le Stab_G(U_1)$ and $Stab_G(U_1)$ is the group of order $q(q-1)$ generated by $\eta$ and $\phi$. On the other hand, $Stab_G(U_1)$ fixes $\ell_1$, hence $|\ell_1^G| = q+1$. Also, since $G$ is transitive on points of $\cC$, we have that through every point of $\cC$ there pass a line of $\ell_1^G$ and, taking into account Lemma \ref{pencil}, we may conclude that $\ell_1^G$ is a regulus of $\cQ_{\lambda}$.

If $\ell_2$ denotes a line tangent to the hyperbolic quadric $\cQ_{\lambda}$ at $P = (0, \sqrt{\lambda}, 0, 1)$, then $P^{\perp_{\lambda}}$ meets $\pi$ in a line, say $r$, that is secant to $\cC$ and $\ell_2 \cap \pi$ is a point $R$ belonging either to $\cI \cap r$ or to $\cE \cap r$. Note that $Stab_G(\ell_2)$ has to fix both the line $r$ and the point $R$; also $Stab_G(r) = \langle \phi, \tau_1 \rangle$. Let $R = (1,0,\alpha,0)$, for some $\alpha \in \GF(q) \setminus \{0\}$. The stabilizer of $R$ in $Stab_G(r)$ is a group $H$ of order $4$ generated by 
$$
\left(\begin{array}{cccc}  1 & 0 & 0 & 0 \\ 0 & -1 & 0 & 0 \\ 0 & 0 & 1 & 0 \\ 0 & 0 & 0 & -1 \end{array}\right) ,  \left(\begin{array}{cccc}  0 & 0 & 1 & 0 \\ 0 & \alpha & 0 & 0 \\ \alpha^2 & 0 & 0 & 0 \\ 0 & 0 & 0 & -\alpha \end{array}\right) .
$$
On the other hand, the unique non--trivial element of $H$ fixing $P$ is an involution. It follows that $|Stab_G(\ell_2)| = 2$. Similarly, if $\lambda$ is a non--square of $\GF(q)$ and $\ell_3$ is a line tangent to the elliptic quadric $\cQ_{\lambda}$ at the point $P' = (-s,0,1,\sqrt{s/{\bar \lambda}})$, for a fixed non--square $s$ in $\GF(q)$, then $P'^{\perp_{\lambda}}$ meets $\pi$ in a line, say $r'$, that is external to $\cC$ and $\ell_3 \cap \pi$ is a point $R'$ belonging either to $\cI \cap r'$ or to $\cE \cap r'$. Again, $Stab_G(\ell_3)$ has to fix both the line $r'$ and the point $R'$; in this case $Stab_G(r') = \langle \psi, \tau_2, \tau_3 \rangle$. The stabilizer of $R'$ in $Stab_G(r')$ is a group $H'$ of order $4$ generated by 
$$
\left(\begin{array}{cccc}  0 & 0 & s^2 & 0 \\ 0 & s & 0 & 0 \\ 1 & 0 & 0 & 0 \\ 0 & 0 & 0 & -s \end{array}\right),  \left(\begin{array}{cccc}  1 & -2\alpha s & s^2\alpha^2 & 0 \\ \alpha & -s\alpha^2-1 & s\alpha  & 0 \\ \alpha^2 & -2\alpha & 1 & 0 \\ 0 & 0 & 0 & s\alpha^2-1 \end{array}\right) ,
$$ 
if $R' = (s \alpha,1,\alpha,0)$, for some $\alpha\in\GF(q)$, or is the group generated by 
$$
\left(\begin{array}{cccc}  0 & 0 & s^2 & 0 \\ 0 & s & 0 & 0 \\ 1 & 0 & 0 & 0 \\ 0 & 0 & 0 & -s \end{array}\right),  \left(\begin{array}{cccc}  1 & 0& 0 & 0 \\ 0 & -1 & 0  & 0 \\ 0& 0&1 & 0 \\ 0 & 0 & 0 & -1\end{array}\right) .
$$
if $R'=(s,0,1,0)$. The unique non--trivial element of $H$ fixing $P$ is an involution. Hence $|Stab_G(\ell_3)|=2$ and  {\em 4)}, {\em 5)} follow.

In order to prove {\em 6)}, let $\ell_4$ be a line tangent to $\cQ_0$ at a point, say $T$, of $\cQ_0 \setminus (\cC \cup \{U_4\})$. Then $\ell_4$ lie in a plane spanned by a line $g$ of $\cQ_0$ and a line of $\pi$ that is tangent to $\cC$, say $t$, with $t \cap g \in \cC$. Since there are $q+1$ of these planes and each such a plane contains $q^2-q$ lines that are tangent to $\cQ_0$ at a point of $\cQ_0 \setminus (\cC \cup \{U_4\})$, it follows that $\ell_4$ can be chosen in $q^3-q$ ways. We will prove that these lines are permuted in a unique orbit under the action of $G$. In order to see this fact, it is enough to show that the stabilizer of $\ell_4$ in $G$ is trivial. Let $T = (1,0,0,1)$ and let $\ell_4 = U_2 T$. Then $g = U_1 U_4$ and $t = U_1 U_2$. The stabilizer of $\ell_4$ in $G$ is contained in the stabilizer of $U_1$ in $Stab_G(U_2) = \langle \phi, \tau_1 \rangle$, that is $Stab_G(\ell_4) \le \phi$. Hence $Stab_G(\ell_4) = Stab_\phi(T)$, that is the identity. 

Finally to prove {\em 7)}, observe that a line of $\cL_4$ is secant to every hyperbolic quadric and external to every elliptic quadric of $\cP$. Therefore, taking into account Lemma \ref{pencil} and Lemma \ref{tang}, it is easily seen that $\cL_4' = \{ r^{\perp_{\lambda}} \;\; | \;\; r \in \cL_4 \}$, for a fixed $\lambda \in \GF(q) \setminus \{0\}$.
\end{proof}
From the proof of Lemma \ref{lines}, we have the following result.
\begin{cor}\label{secant}
A line of $\cL_4$ is secant to every hyperbolic quadric and external to every elliptic quadric of $\cP$. A line of $\cL_4'$ is secant to every quadric of $\cP$. 
\end{cor}

\begin{defin} 
A point $R$ of $\PG(3,q)$, with $R \not\in \pi \cup \{U_4\}$, is said to be {\em of type $\cI$} (resp. {\em of type $\cE$}, {\em of type $\cC$}), if the line joining $U_4$ with $R$ meets $\pi$ in a point of $\cI$ (resp. $\cE$, $\cC$).   
\end{defin}

\begin{lemma}\label{lemma}
\begin{itemize}
\item[i)] Through the point $U_4$ there pass: $q+1$ lines of $\cL_1'$, $q(q-1)/2$ lines of $\cL_2'$, $q(q+1)/2$ lines of $\cL_3'$ and no line of $\cL_1$, $\cL_2$, $\cL_3$, $\cL_4$, $\cL_4'$;
\item[ii)] through a point of $\cC$ there pass: one line of $\cL_1$, $q$ lines of $\cL_3$, one line of $\cL_1'$, $q(q-1)$ lines of $\cL_4'$ and no line of $\cL_2$, $\cL_4$, $\cL_2'$, $\cL_3'$;
\item[iii)] through a point of $\cI$ there pass: one line of $\cL_2'$, $(q+1)/2$ lines of $\cL_2$, $(q+1)/2$ lines of $\cL_3$ and no line of $\cL_1$, $\cL_4$, $\cL_1'$, $\cL_3'$, $\cL_4'$;
\item[iv)] through a point of $\cE$ there pass: $2$ lines of $\cL_1$, $(q-1)/2$ lines of $\cL_2$, $(q-1)/2$ lines of $\cL_3$, $2(q-1)$ lines of $\cL_4$, one line of $\cL_3'$ and no line of $\cL_1'$, $\cL_2'$, $\cL_4'$;
\item[v)] through a point of type $\cC$ there pass: one line of $\cL_1'$, $q$ lines of $\cL_4'$, $q$ lines of $\cL_4$ and no line of $\cL_1$, $\cL_2$, $\cL_3$, $\cL_2'$, $\cL_3'$,
\item[vi)] through a point of type $\cI$ there pass: one line of $\cL_2'$, $q+1$ lines of $\cL_4'$ and no line of $\cL_1$, $\cL_2$, $\cL_3$, $\cL_4$, $\cL_1'$, $\cL_3'$,
\item[vii)] through a point of type $\cE$ there pass: one line of $\cL_3'$, $q-1$ lines of $\cL_4'$, $2(q-1)$ lines of $\cL_4$ and no line of $\cL_1$, $\cL_2$, $\cL_3$, $\cL_1'$, $\cL_2'$.
\end{itemize}
\end{lemma}
\begin{proof}
Preliminarily note that a line of $\cL_1'$ or of $\cL_4'$ meets $\pi$ in a point of $\cC$; a line of $\cL_3'$ or of $\cL_4$ meets $\pi$ in a point of $\cE$ and a line of $\cL_2'$ intersects $\pi$ in a point of $\cI$. Also, a line of $\cL_1'$ (resp. $\cL_2'$, $\cL_3'$) contains $q-1$ points of type $\cC$ (resp. of type $\cI$, of type $\cE$). Taking into account Corollary \ref{cono} and Corollary \ref{secant}, it follows that a line of $\cL_4$ contains $q-1$ points of type $\cE$ and one point of type $\cC$, whereas a line of $\cL_4'$ contains $(q-1)/2$ points of type $\cE$, $(q-1)/2$ points of type $\cI$ and one point of type $\cC$.

It is straightforward to prove $i)$. To prove $ii)$, observe that $|\cC| = q+1$, $|\cL_4'| = q^3-q$ and every line of $\cL_4'$ contains exactly one point of $\cC$, see Lemma \ref{lines}, {\em 7)}. Hence through a point of $\cC$ there pass $q^2-q$ lines of $\cL_4'$. Let $P \in \PG(3,q) \setminus (\cC \cup \{U_4\})$. 

\medskip
\fbox{$P \in \cI$}
\medskip

\noindent The unique line of $\cL_2'$ through $P$ is the line $U_4 P$. Also among the $q+1$ lines of $\pi$ through $P$ there are $(q+1)/2$ that are secant to $\cC$ and $(q+1)/2$ that are external to $\cC$. 

\medskip
\fbox{$P \in \cE$}
\medskip

\noindent Among the $q+1$ lines of $\pi$ through $P$ there are two lines of $\cL_1$, $(q-1)/2$ lines of $\cL_2$ and $(q-1)/2$ lines of $\cL_3$. The unique line of $\cL_3'$ through $P$ is the line $U_4 P$. Moreover, since a line of $\cL_4$ contains exactly one point of $\cE$, $|\cE| = (q^2+q)/2$ and $|\cL_4| = q^3-q$, we have that through the point $P$ there pass $2(q-1)$ lines of $\cL_4$. 

\medskip
\fbox{$P$ point of type $\cC$}
\medskip

\noindent The unique line of $\cL_1'$ through $P$ is the line $g = U_4 P$. Let $Z = \cC \cap g$. Then the lines of $\cL_4'$ through $P$ are those joining $P$ with a point of $\cC \setminus \{Z\}$. Similarly if $z$ denotes the line of $\pi$ that is tangent to $\cC$ at the point $Z$, then the lines of $\cL_4$ through $P$ are those joining $P$ with a point of $z \setminus \{Z\}$.

\medskip
\fbox{$P$ point of type $\cI$}
\medskip

\noindent The unique line of $\cL_2'$ through $P$ is the line $U_4 P$, whereas the $q+1$ lines of $\cL_4'$ through $P$ are those joining $P$ with a point of $\cC$.

\medskip
\fbox{$P$ point of type $\cE$}
\medskip

\noindent The unique line of $\cL_3'$ through $P$ is the line $g = U_4 P$. Let $Y = g \cap \pi \in \cE$. Through the point $Y$ there pass two lines of $\pi$, say $r_1$, $r_2$, that are tangent to $\cC$ at the point $R_1$, $R_2$, respectively. Then the lines of $\cL_4'$ through $P$ are those joining $P$ with a point of $\cC \setminus \{R_1, R_2\}$, whereas the the lines of $\cL_4$ containing $P$ are those joining $P$ with the points of $(r_1 \cup r_2) \setminus \{U_4, R_1, R_2\}$.
\end{proof}
Dually we have the following result.
\begin{lemma}\label{lemma1}
Let $\sigma$ be a plane of $\PG(3,q)$ and let $r = \sigma \cap \pi$.
\begin{enumerate}
\item[i)] If $\sigma = \pi$, then $\sigma$ contains: $q+1$ lines of $\cL_1$, $q(q-1)/2$ lines of $\cL_2$, $q(q+1)/2$ line of $\cL_3$ and no line of $\cL_1'$, $\cL_2'$, $\cL_3'$, $\cL_4'$, $\cL_4$.
\item[] If $U_4 \in \sigma$ and  
\begin{itemize}
\item[ii)] $|r \cap \cC| = 0$, then $\sigma$ contains: one line of $\cL_2$, $(q+1)/2$ lines of $\cL_2'$, $(q+1)/2$ lines of $\cL_3'$ and no line of $\cL_1$, $\cL_1'$, $\cL_3$, $\cL_4$, $\cL_4'$;
\item[iii)] $|r \cap \cC| = 1$, then $\sigma$ contains: one line of $\cL_1'$, $q$ lines of $\cL_3'$, one line of $\cL_1$, $q(q-1)$ lines of $\cL_4$ and no line of $\cL_2$, $\cL_3$, $\cL_2'$, $\cL_4'$;
\item[iv)] $|r \cap \cC| = 2$, then $\sigma$ contains: $2$ lines of $\cL_1'$, $(q-1)/2$ lines of $\cL_2'$, $(q-1)/2$ lines of $\cL_3'$, $2(q-1)$ lines of $\cL_4'$, one line of $\cL_3$ and no line of $\cL_1$, $\cL_2$, $\cL_4$.
\end{itemize}
\item[] If $U_4 \notin \sigma$ and 
\begin{itemize}
\item[v)] $|r \cap \cC| = 0$, then $\sigma$ contains: one line of $\cL_2$, $q+1$ lines of $\cL_4$ and no line of $\cL_1'$, $\cL_2'$, $\cL_3'$, $\cL_4'$, $\cL_1$, $\cL_3$, 
\item[vi)] $|r \cap \cC| = 1$, then $\sigma$ contains: one line of $\cL_1$, $q$ lines of $\cL_4$, $q$ lines of $\cL_4'$ and no line of $\cL_1'$, $\cL_2'$, $\cL_3'$, $\cL_2$, $\cL_3$,
\item[vii)] $|r \cap \cC| = 2$, then $\sigma$ contains: one line of $\cL_3$, $q-1$ lines of $\cL_4$, $2(q-1)$ lines of $\cL_4'$ and no line of $\cL_1'$, $\cL_2'$, $\cL_3'$, $\cL_1$, $\cL_2$.
\end{itemize}
\end{enumerate}
\end{lemma}
\begin{proof}
First of all observe that a line of $\cL_1'$ or of $\cL_2'$ or of $\cL_3'$ passes through $U_4$; also, a line of $\cL_4$ is contained in a plane spanned by a line of the quadratic cone $\cQ_0$ and a line of $\pi$ that is tangent to $\cC$. 

Let $\sigma$ be a plane of $\PG(3,q)$. If $\sigma = \pi$, then the proof easily follows. Let $\sigma \ne \pi$ and let $r = \sigma \cap \pi$. 

Assume first that $U_4 \in \sigma$. There are three possibilities according as $|r \cap \cC|$ is $0,1,2$. If $|r \cap \cC| = 0$, then $r \in \cL_2$. Also among the $q+1$ lines of $\sigma$ through $U_4$ there are $(q+1)/2$ lines of $\cL_2'$ and $(q+1)/2$ lines of $\cL_3'$. If $r \in \cL_1$, then $\sigma$ contains one line of $\cL_1'$, that is the line joining $U_4$ with $r \cap \cC$ and $q$ lines of $\cL_3'$. The $q^2-q$ lines of $\sigma$ distinct from $r$ and not containing $r \cap \cC$ and $U_4$ are lines of $\cL_4$. If $r \in \cL_3$, then among the $q+1$ lines of $\sigma$ through $U_4$ there are two lines of $\cL_1'$, $(q-1)/2$ lines of $\cL_2'$ and $(q-1)/2$ lines of $\cL_3'$. Moreover the lines of $\cL_4'$ contained in $\sigma$ are those passing through one of the two points $r \cap \cC$, distinct from $r$ and not containing $U_4$.

Assume now that $U_4 \notin \sigma$. Again three possibilities arise according as $|r \cap \cC|$ is $0,1,2$. If $|r \cap \cC| = 0$, then $r \in \cL_2$ and $\sigma \cap \cQ_0$ is a non--degenerate conic having no point in common with $\pi$. It follows that the lines of $\cL_4$ contained in $\sigma$ are the $q+1$ lines meeting $\sigma \cap \cQ_0$ in one point. If $|r \cap \cC| = 1$, then $r \in \cL_1$ and $\sigma \cap \cQ_\lambda$ is a non--degenerate conic meeting $\pi$ in $r \cap \cC$, $\lambda \in \GF(q)$. Hence the $q$ lines of $\cL_4'$ contained in $\sigma$ are those through $r \cap \cC$ and meeting $\pi$ exactly in $r \cap \cC$, whereas the $q$ lines of $\cL_4$ contained in $\sigma$ are the lines that are tangent to $\sigma \cap \cQ_0$ at a point distinct from $r \cap \pi$. Finally, if $r \cap \cC = \{R_1, R_2\}$, with $R_1 \ne R_2$, then $r \in \cL_3$ and there exists a non--zero square element of $\GF(q)$, say ${\bar \lambda}$, such that $\sigma \cap \cQ_{\bar \lambda}$ consists of two lines, say $r_1$, $r_2$, with $r_i \cap \pi = R_i$, $i=1,2$. If $\lambda \in \GF(q) \setminus \{{\bar \lambda}\}$, then $\sigma \cap \cQ_{\lambda}$ is a non--degenerate conic passing through $R_1$ and $R_2$. In this case the line $r_i \subset \sigma$ is tangent to $\sigma \cap \cQ_{\lambda}$ at the point $R_i$, $\lambda \in \GF(q) \setminus \{{\bar \lambda}\}$. It follows that the $2(q-1)$ lines of $\cL_4'$ contained in $\sigma$ are those through $R_i$, distinct from $r_i$ and meeting $\pi$ exactly in $R_i$, $i = 1,2$, whereas the $q-1$ lines of $\cL_4$ contained in $\sigma$ are the lines that are tangent to $\sigma \cap \cQ_0$ at a point distinct from $R_1, R_2$.  
\end{proof}

\subsection{A derivation technique}

Let $\cA$, $\cB$ be two line--sets of $\PG(3,q)$. For a line $\ell$ of $\PG(3,q)$, we consider the following sets: 
$$
\cA_{\ell} = \{ r \in \cA \; : \; | r \cap \ell | \ge 1\}, \; \cB_{\ell} = \{ r \in \cB \; : \; | r \cap \ell | \ge 1\}.
$$

	\begin{theorem}\label{cl}
Let $\cL$ be a Cameron--Liebler line class with parameter $(q^2+1)/2$ and let $\cA$, $\cB$ be line sets of equal size of $\PG(3,q)$ such that:
\begin{itemize}
\item[i)] $\cA \subset \cL$ and $|\cB \cap \cL| = 0$,
\item[ii)] if $\ell \notin \cA \cup \cB$, then $|\cA_{\ell}| = |\cB_{\ell}|$,
\item[iii)] if $\ell \in \cA$, then $|\cA_\ell| - |\cB_{\ell}| = q^2$,
\item[iv)] if $\ell \in \cB$, then $|\cB_\ell| - |\cA_{\ell}| = q^2$.
\end{itemize}
Then the set $\bar{\cL} = (\cL \setminus \cA) \cup \cB$ is a Cameron--Liebler line class with parameter $(q^2+1)/2$.
	\end{theorem} 
\begin{proof}
Since $\cL$ is a Cameron--Liebler line class with parameter $(q^2+1)/2$, we have that $|\{r \in \cL \; : \; |r \cap \ell| \ge 1\}|$ equals  $q^2+ (q+1)(q^2+1)/2$ if $\ell \in \cL$, or $(q + 1)(q^2+1)/2$ if $\ell \notin \cL$. 

Let $\ell$ be a line of $\PG(3,q)$. 
\begin{itemize}
\item  If $\ell \in \cL \setminus (\cA \cup \cB)$, then $\ell \in \bar{\cL}$. From $ii)$, it follows that $|\{r \in \bar{\cL} \; : \; |r \cap \ell| \ge 1\}|$ equals  $q^2+ (q+1)(q^2+1)/2$. 
\item If $\ell \notin \cL \cup \cA \cup \cB$, then $\ell \notin \bar{\cL}$. From $ii)$, it follows that $|\{r \in \bar{\cL} \; : \; |r \cap \ell| \ge 1\}|$ equals $(q + 1)(q^2+1)/2$. 
\item If $\ell \in \cA$, then $\ell \in \cL \setminus \bar{\cL}$. From $iii)$, we have that $|\{r \in \bar{\cL} \; : \; |r \cap \ell| \ge 1\}|$ equals $q^2 + (q + 1)(q^2+1)/2 - |\cA_\ell| + |\cB_{\ell}| = (q+1)(q^2+1)/2$.
\item If $\ell \in \cB$, then $\ell \in \bar{\cL} \setminus \cL$. From $iv)$, we have that $|\{r \in \bar{\cL} \; : \; |r \cap \ell| \ge 1\}|$ equals $(q + 1)(q^2+1)/2 + |\cB_\ell| - |\cA_\ell| = q^2 + (q+1)(q^2+1)/2$. 
\end{itemize}
The proof is now complete.
\end{proof}

With the same notation introduced in Section \ref{setting}, consider the following two line--sets of $\PG(3,q)$: 
$$
\cA := \cL_1' \cup \cL_2' \cup \cL_3 \cup \cL_4' , \;\; \cB := \cL_1 \cup \cL_2 \cup \cL_3' \cup \cL_4 .
$$
Set $n_i := |\{r : r \in \cL_i \;\; : \;\; |r \cap \ell| \ge 1\}|$ and $n_i' := |\{r : r \in \cL_i' \;\; : \;\; |r \cap \ell| \ge 1\}|$, $1 \le i \le 4$, where $\ell$ is a line of $\PG(3,q)$. With the notation introduced in the previous section, we have that 
$$
|\cA_\ell| = n_1' + n_2' + n_3 + n_4' , \;\; |\cB_\ell| = n_1 + n_2 + n_3' + n_4. 
$$
We will prove that if $q \equiv 1 \pmod{4}$, then there exists a Cameron--Liebler line class of Bruen--Drudge type such that the sets $\cA$, $\cB$ satisfy the hypotheses of Theorem \ref{cl}.  

\begin{prop}\label{prop1}
Let $\ell$ be a line of $\PG(3,q)$, if $\ell \not \in \cA \cup \cB$, then $|\cA_{\ell}| = |\cB_{\ell}|$.
\end{prop}
\begin{proof}
Let $\ell$ be a line of $\PG(3,q)$, with $\ell \not\in \cA \cup \cB$. From Lemma \ref{lines}, we have three possibilities: either $\ell$ is contained in a regulus of a hyperbolic quadric $\cQ_{\lambda}$, for some non--zero square $\lambda \in \GF(q)$, or $\ell$ is tangent to a quadric $\cQ_{\lambda}$, $\lambda \ne 0$, and $\ell \cap \pi$ is a point of $\cI$ or $\ell$ is tangent to a quadric $\cQ_{\lambda}$, $\lambda \ne 0$, and $\ell \cap \pi$ is a point of $\cE$. Let $\ell'$ be the line obtained by intersecting $\pi$ with the plane spanned by $U_4$ and $\ell$.

If the first case occurs, then $\ell'$ is tangent to $\cC$. Hence $\ell$ contains $q$ points of type $\cE$ and one point of $\cC$. Taking into account Lemma \ref{lemma}, we get $n_1 = 1$, $n_2 = 0$, $n_3 = q$, $n_4 = 2(q^2-q)$, $n_1' = 1$, $n_2' = 0$, $n_3' = q$ and $n_4' = 2(q^2-q)$. Hence $|\cA_\ell| = |\cB_\ell| = 2q^2-q+1$.    

Assume that $\ell$ is tangent to a quadric $\cQ_\lambda$, $\lambda \ne 0$, and $\ell \cap \pi$ is a point of $\cI$ or of $\cE$. Then $\ell'$ is secant or external to $\cC$. If $\ell \cap \pi \in \cI$ and $\ell'$ is secant to $\cC$, then $\ell$ contains $(q-1)/2$ points of type $\cE$, $(q-3)/2$ points of type $\cI$, two points of type $\cC$ and one point of $\cI$. In this case we have $n_1 = 0$, $n_2 = (q+1)/2$, $n_3 = (q+1)/2$, $n_4 = q^2+1$, $n_1' = 2$, $n_2' = (q-1)/2$, $n_3' = (q-1)/2$, $n_4' = q^2-1$ and hence $|\cA_\ell| = |\cB_\ell| = q^2+q+1$.

If $\ell \cap \pi \in \cI$ and $\ell'$ is external to $\cC$, then $\ell$ contains $(q+1)/2$ points of type $\cE$, $(q-1)/2$ points of type $\cI$ and one point of $\cI$. In this case we obtain $n_1 = 0$, $n_2 = (q+1)/2$, $n_3 = (q+1)/2$, $n_4 = q^2-1$, $n_1' = 0$, $n_2' = (q+1)/2$, $n_3' = (q+1)/2$, $n_4' = q^2-1$ and hence $|\cA_\ell| = |\cB_\ell| = q^2+q$.

If $\ell \cap \pi \in \cE$ and $\ell'$ is secant to $\cC$, then $\ell$ contains $(q-1)/2$ points of type $\cI$, $(q-3)/2$ points of type $\cE$, two points of type $\cC$ and one point of $\cE$. In this case we get $n_1 = 2$, $n_2 = (q-1)/2$, $n_3 = (q-1)/2$, $n_4 = q^2+1$, $n_1' = 2$, $n_2' = (q-1)/2$, $n_3' = (q-1)/2$, $n_4' = q^2+1$ and hence $|\cA_\ell| = |\cB_\ell| = q^2+q+2$.

If $\ell \cap \pi \in \cE$ and $\ell'$ is external to $\cC$, then $\ell$ contains $(q+1)/2$ points of type $\cI$, $(q-1)/2$ points of type $\cE$ and one point of $\cE$. In this case we have $n_1 = 2$, $n_2 = (q-1)/2$, $n_3 = (q-1)/2$, $n_4 = q^2-1$, $n_1' = 0$, $n_2' = (q+1)/2$, $n_3' = (q+1)/2$, $n_4' = q^2+1$ and hence $|\cA_\ell| = |\cB_\ell| = q^2+q+1$.    
\end{proof}

\begin{prop}\label{prop2}
Let $\ell$ be a line of $\PG(3,q)$.
\begin{itemize} 
\item If $\ell \in \cA$, then $|\cA_\ell| - |\cB_{\ell}| = q^2$,
\item If $\ell \in \cB$, then $|\cB_\ell| - |\cA_{\ell}| = q^2$.
\end{itemize}
\end{prop}
\begin{proof}
If $\ell \in \cL_1$, then $\ell$ consists of $q$ points of $\cE$ and one point of $\cC$. Taking into account Lemma~\ref{lemma}, we have that $n_1 = q$, $n_2 = q(q-1)/2$, $n_3 = q(q+1)/2$, $n_4 = 2(q^2-q)$, $n_1' = 1$, $n_2' = 0$, $n_3' = q$, $n_4' = q^2-q$. Hence $|\cA_\ell| = (3q^2-q+2)/2$ and $|\cB_\ell| = q^2+(3q^2-q+2)/2$.

If $\ell \in \cL_2$, then $\ell$ contains $(q+1)/2$ points of $\cE$ and $(q+1)/2$ points of $\cI$. In this case it turns out that $n_1 = q+1$, $n_2 = q(q-1)/2$, $n_3 = q(q+1)/2$, $n_4 = q^2-1$, $n_1' = 0$, $n_2' = (q+1)/2$, $n_3' = (q+1)/2$, $n_4' = 0$ and hence $|\cA_\ell| = (q+1)^2/2$ and $|\cB_\ell| = q^2+(q+1)^2/2$.

If $\ell \in \cL_3$, then $\ell$ has $(q-1)/2$ points of $\cE$, $(q-1)/2$ points of $\cI$ and two points of $\cC$. In this case it follows that $n_1 = q+1$, $n_2 = q(q-1)/2$, $n_3 = q(q+1)/2$, $n_4 = (q-1)^2$, $n_1' = 2$, $n_2' = (q-1)/2$, $n_3' = (q-1)/2$, $n_4' = 2(q^2-q)$ and hence $|\cA_\ell| = q^2 + (3q^2-2q+3)/2$ and $|\cB_\ell| = (3q^2-2q+3)/2$.

If $\ell \in \cL_4$, then in $\ell$ there are $q-1$ points of type $\cE$, one point of type $\cC$ and one point of $\cE$. In this case we obtain $n_1 = 2$, $n_2 = (q-1)/2$, $n_3 = (q-1)/2$, $n_4 = 2(q^2-q)$, $n_1' = 1$, $n_2' = 0$, $n_3' = q$, $n_4' = q^2-q+1$ and hence $|\cA_\ell| = (2q^2-q+3)/2$ and $|\cB_\ell| = q^2+(2q^2-q+3)/2$.

If $\ell \in \cL_1'$, then, apart from $U_4$, $\ell$ consists of $q-1$ points of type $\cC$ and one point of $\cC$. In this case we have that $n_1 = 1$, $n_2 = 0$, $n_3 = q$, $n_4 = q^2-q$, $n_1' = q+1$, $n_2' = q(q-1)/2$, $n_3' = q(q+1)/2$, $n_4' = 2(q^2-q)$. Hence $|\cA_\ell| = q^2 + (3q^2-q+2)/2$ and $|\cB_\ell| = (3q^2-q+2)/2$.

If $\ell \in \cL_2'$, then, apart from $U_4$, $\ell$ contains $q-1$ points of type $\cI$ and one point of $\cI$. In this case it turns out that $n_1 = 0$, $n_2 = (q+1)/2$, $n_3 = (q+1)/2$, $n_4 = 0$, $n_1' = q+1$, $n_2' = q(q-1)/2$, $n_3' = q(q+1)/2$, $n_4' = q^2-1$ and hence $|\cA_\ell| = q^2 + (q+1)^2/2$ and $|\cB_\ell| = (q+1)^2/2$.

If $\ell \in \cL_3'$, then, apart from $U_4$, $\ell$ has $q-1$ points of type $\cE$ and one point of $\cE$. In this case it follows that $n_1 = 2$, $n_2 = (q-1)/2$, $n_3 = (q-1)/2$, $n_4 = 2(q^2-q)$, $n_1' = q+1$, $n_2' = q(q-1)/2$, $n_3' = q(q+1)/2$, $n_4' = (q-1)^2$ and hence $|\cA_\ell| = (3q^2-2q+3)/2$ and $|\cB_\ell| = q^2 + (3q^2-2q+3)/2$.

If $\ell \in \cL_4'$, then, apart from $U_4$, in $\ell$ there are $(q-1)/2$ points of type $\cI$, $(q-1)/2$ points of type $\cE$, one point of type $\cC$ and one point of $\cC$. In this case we obtain $n_1 = 1$, $n_2 = 0$, $n_3 = q$, $n_4 = q^2-q+1$, $n_1' = 2$, $n_2' = (q-1)/2$, $n_3' = (q-1)/2$, $n_4' = 2(q^2-q)$ and hence $|\cA_\ell| = q^2 + (2q^2-q+3)/2$ and $|\cB_\ell| = (2q^2-q+3)/2$.
\end{proof}

We are ready to prove the main result of this section.

\begin{theorem}\label{main}
If $q \equiv 1 \pmod{4}$, then there exists a Cameron--Liebler line class $\cL$ such that $\cA \subset \cL$ and $|\cB \cap \cL| = 0$. 
\end{theorem}
\begin{proof}
Let ${\bar \lambda}$ be a fixed non--square element of $\GF(q)$ and consider the elliptic quadric $\cQ_{\bar \lambda}$ of $\cP$. Let $F$ be the quadratic form defining $\cQ_{\bar \lambda}$ and denote with $O_s$, $O_n$ the points of $\PG(3,q)$ such that the evaluation of the quadratic form $F$ is a non--zero square or a non--square in $\GF(q)$, respectively. Recall that a line that is tangent to $\cQ_{\bar \lambda}$ contains either $q$ points of $O_s$ or $q$ points of $O_n$. Let $\cL$ be the Cameron--Liebler line class of $\PG(3,q)$ obtained by joining the tangent lines to $\cQ_{\bar \lambda}$ containing $q$ points of $O_n$, say $\cT$, and the secant lines to $\cQ_{\bar \lambda}$, say $\cS$. Note that $U_4 \in O_n$, since $F(U_4) = {\bar \lambda}$; on the other hand $\cE \subset O_s$ if and only if $q \equiv 1 \pmod{4}$. To see this fact it is enough to show that the point $U_2 \in \cE$ belongs to $\cO_s$ if and only if $q \equiv 1 \pmod{4}$, see Remark \eqref{tra}, which holds true since $F(U_2) = -1$. It follows that, if $q \equiv 1 \pmod{4}$, the set $\cT$ contains $\cL_1'$, $(q-1)/2$ orbits of type {\em 3)} and one orbit of type {\em 4)}, whereas the set $\cS$ contains $\cL_2'$, $\cL_3$, $\cL_4'$, $(q-3)/2$ orbits of type {\em 4)} and $(q-1)/2$ orbits of type {\em 5)}. In particular $\cA = \cL_1' \cup \cL_2' \cup \cL_3 \cup \cL_4'$ is contained in $\cL$ and $\cB = \cL_1 \cup \cL_2 \cup \cL_3' \cup \cL_4$ is disjoint from $\cL$, as required.    
\end{proof}

\section{The isomorphism problem}

In this section we assume $q \equiv 1 \pmod{4}$ and we keep the notation used above. According to Theorem \ref{main}, there exists a Cameron--Liebler line class $\cL$ such that $\cA \subset \cL$ and $|\cB \cap \cL| = 0$. In particular $\cL = \cT \cup \cS$, where $\cT$ is the set of tangent lines to $\cQ_{\bar \lambda}$ containing $q$ points of $O_n$ and $\cS$ is the set of secant lines to $\cQ_{\bar \lambda}$. Here $\cQ_{\bar \lambda}$ is a fixed elliptic quadric of $\cP$. Hence $\cL$ is equivalent to a Cameron--Liebler line class constructed by Bruen and Drudge. 

\begin{remark}\label{oss}
There are either $q^2 + (q+1)/2$ or $(q^2+q)/2+1$ or $(q^2+q)/2+q$ lines of $\cL$ through a point $P$ of $\PG(3,q)$, according as $P \in \cQ_{\bar \lambda}$, $O_n$ or $O_s$, respectively. On the other hand, a plane $\sigma = P^{\perp_{\bar \lambda}}$ of $\PG(3,q)$ contains either $(q+1)/2$ or $(q^2+q)/2$ or $(q^2+q)/2+q+1$, according as $P \in \cQ_{\bar \lambda}$, $O_n$ or $O_s$, respectively. Note that $|P^\perp_{\bar \lambda} \cap \cC| = 0,1,2$, depending on $P$ is a point of type $\cE, \cC, \cI$, respectively.
\end{remark}

From Proposition \ref{prop1}, Proposition \ref{prop2}, Theorem \ref{main} and Theorem \ref{cl}, we have that ${\bar \cL} = \cL \setminus \cA \cup \cB$ is a Cameron--Liebler line class with parameter $(q^2+1)/2$. The main goal of this section will be to show that $\cL$ is not isomorphic to one of the previously known examples of Cameron--Liebler line classes with parameter $(q^2+1)/2$, whenever $q \ge 9$. In order to establish this result we investigate the characters of ${\bar \cL}$ with respect to line--stars and line--sets in planes of $\PG(3,q)$.  

\begin{prop}\label{char1}
The characters of ${\bar \cL}$, with respect to line--stars of $\PG(3,q)$ are:
$$
\frac{q+1}{2}, \frac{q^2+q}{2}-2(q+1), \frac{q^2+q}{2}-(q+1), \frac{q^2+q}{2}, \frac{q^2+q}{2}+q+1, q^2 - \frac{q+3}{2}.
$$
\end{prop}
\begin{proof}
From the proof of Theorem \ref{main}, we have that $\cE \subset O_s$, $\cI \subset O_n$ and $U_4 \in O_n$. Let $P$ be a point of $\PG(3,q)$. Note that if $x$ is the number of lines of $\cL$ through the point $P$, then there are exactly $y$ lines of ${\bar \cL}$ containing $P$, where 
\begin{equation}\label{eq1}
y = x - m_1' - m_2' - m_3 - m_4' + m_1 + m_2 + m_3' + m_4. 
\end{equation}
Here $m_i$, $m_i'$ denote the number of lines of $\cL_i$, $\cL_i'$, respectively, through the point $P$, $1 \le i \le 4$. Taking into account \eqref{eq1}, Remark \eqref{oss} and Lemma \ref{lemma}, we have that if $P = U_4$, then $y = (q^2+q)/2$. Analogously, if $P \in \cC$, $\cE$, $\cI$, then $y = (q+1)/2$, $(q^2+q)/2+q+1$, $(q^2+q)/2$, respectively. 

Assume that $P$ is a point of type $\cI$. Then there are three possibilities: either $P$ belongs to $\cQ_{\bar \lambda}$ and there are $q^2-q$ such points, or $P \in O_n$ and there are $q(q-1)(q-5)/4$ such points, or $P \in O_s$ and there are $q(q-1)^2/4$ such points. It turn out that if $P \in \cQ_{\bar \lambda}$, $O_n$, $O_s$, then $y = q^2-(q+3)/2$, $(q^2+q)/2-(q+1)$, $(q^2+q)/2-2(q+1)$, respectively.     

Assume that $P$ is a point of type $\cE$. Then there are two possibilities: either $P$ belongs to $O_n$ or to $O_s$. In both cases there are $(q^3-q)/4$ such points. It follows that if $P \in O_n$, then $y = (q^2+q)/2+(q+1)$, whereas if $P \in O_s$, then $y = (q^2+q)/2$.

Assume that $P$ is a point of type $\cC$. Then $P \in O_n$ and $y = (q^2+q)/2$.
\end{proof}

\begin{prop}\label{char2}
The characters of ${\bar \cL}$, with respect to line--sets in planes of $\PG(3,q)$ are:
$$
\frac{3q+5}{2}, \frac{q^2+q}{2}-q, \frac{q^2+q}{2}+1, \frac{q^2+q}{2}+q+2, \frac{q^2+q}{2}+2q+3, q^2 + \frac{q+1}{2}.
$$
\end{prop}
\begin{proof}
From the proof of Theorem \ref{main}, we have that $\cE \subset O_s$, $\cI \subset O_n$ and $U_4 \in O_n$. Let $\sigma = P^{\perp_{\bar \lambda}}$ be a plane of $\PG(3,q)$. Note that if $x$ is the number of lines of $\cL$ contained in $\sigma$, then there are exactly $y$ lines of ${\bar \cL}$ contained in $\sigma$, where 
\begin{equation}\label{eq2}
y = x - t_1' - t_2' - t_3 - t_4' + t_1 + t_2 + t_3' + t_4. 
\end{equation}
Here $t_i$, $t_i'$ denote the number of lines of $\cL_i$, $\cL_i'$, respectively, contained in $\sigma$, $1 \le i \le 4$. Taking into account \eqref{eq2}, Remark \eqref{oss} and Lemma \ref{lemma1}, we have that if $P = U_4$, i.e. $\sigma = \pi$, then $y = (q^2+q)/2+1$. Analogously, if $P \in \cC$, $\cE$, $\cI$, then $\sigma$ contains $U_4$ and meets $\cC$ in $1$, $2$ or no point, respectively. Thus $y = q^2+ (q+1)/2$, $(q^2+q)/2-q$, $(q^2+q)/2+1$, respectively. 

Assume that $P$ is a point of type $\cI$. Then $U_4 \not\in \sigma$ and $|\sigma \cap \cC| = 0$. There are three possibilities: either $P$ belongs to $\cQ_{\bar \lambda}$ or $P \in O_n$ or $P \in O_s$. It turn out that if $P \in \cQ_{\bar \lambda}$, $O_n$, $O_s$, then $y = (3q+5)/2$, $(q^2+q)/2+q+2$, $(q^2+q)/2+2q+3$, respectively.     

Assume that $P$ is a point of type $\cE$. Then $U_4 \not\in \sigma$ and $|\sigma \cap \cC| = 2$. There are two possibilities: either $P$ belongs to $O_n$ or to $O_s$. It follows that if $P \in O_n$, then $y = (q^2+q)/2-q$, whereas if $P \in O_s$, then $y = (q^2+q)/2+1$.

Assume that $P$ is a point of type $\cC$. Then $U_4 \not\in \sigma$, $|\sigma \cap \cC| = 1$ and $P \in O_n$. Thus $y = (q^2+q)/2+1$.
\end{proof}

\begin{theorem}
If $q \ge 9$, the Cameron--Liebler line class ${\bar \cL}$ is not equivalent to one of the previously known examples.
\end{theorem}
\begin{proof}
From the proof of Proposition \ref{char2}, if $P \in \cQ_{\bar \lambda} \setminus \cC$, through $P$ there pass $(3q+5)/2$ lines of $\bar \cL$. Since, for $q \ge 9$, the value $(3q+5)/2$ does not appear among the characters of $\cL'$, $\cL''$ and $\cL'''$, we may conclude that $\bar \cL$ is distinct from $\cL'$, $\cL''$ and $\cL'''$ (see Section \ref{known}). On the other hand, both examples $\cX \cup \cY$ and $\cX \cup {\cal Z}$ admit $q^2+q+1$ as a character, but from Proposition \ref{char1} and Proposition \ref{char2}, such a value does not appear as a character of $\bar \cL$.   
\end{proof}


\end{document}